\documentclass[12pt, reqno]{amsart}
\usepackage{amsthm}
\usepackage{amsfonts}
\usepackage{amsmath}
\usepackage{mathrsfs}
\usepackage{amssymb}
\usepackage{amsfonts}
\usepackage{amsbsy}

\usepackage{CJK,CJKnumb}
\usepackage{color}
\usepackage{indentfirst}
\usepackage{latexsym,bm}
\usepackage{graphicx}
\usepackage{cases}
\usepackage{pifont}
\usepackage{txfonts}
\usepackage{xcolor}
\usepackage[all,knot,poly]{xy}

\usepackage[labelformat=empty]{caption}
\usepackage{enumitem}

\newtheorem{theo}{Theorem}[section]
\newtheorem{lem}[theo]{Lemma}
\newtheorem{coro}[theo]{Corollary}
\newtheorem{prop}[theo]{Proposition}
\newtheorem{defi}[theo]{Definition}

\theoremstyle{rem}
\newtheorem{rem}[theo]{Remark}
\theoremstyle{ex}

\numberwithin{equation}{section}

\newcommand{\C}{{\mathbb C}}

\newcommand{\Z}{{\mathbb Z}}

\newcommand{\osp}{{\rm\mathfrak{osp}}}
\newcommand{\Sl}{{\rm\mathfrak{sl}}}

\newcommand{\g}{{\mathfrak g}}

\newcommand{\de}{\delta}
\newcommand{\ga}{\gamma}
\newcommand{\ve}{\varepsilon}

\newcommand{\Q}{{\mathcal Q}}
\newcommand{\hQ}{{\widehat{\mathcal Q}}}

\newcommand{\cq}{{\mathcal{C}_q}}

\newcommand{\Uq}{{{\rm  U}_q}}
\newcommand{\U}{{\rm  U }}
\newcommand{\Dr}{{\rm U}^D}

\newcommand{\D}{\mathfrak{U}^D}
\newcommand{\UU}{{\mathfrak U}}

\newcommand{\cI}{{\mathcal I}}

\newcommand{\e}{\xi^+ }
\newcommand{\f}{\xi^-}
\newcommand{\ef}{\xi^{\pm}}
\newcommand{\qe}{e}
\newcommand{\qf}{f}
\newcommand{\x}{\xi'}
\newcommand{\h}{\kappa}
\newcommand{\qk}{\gamma}
\newcommand{\hh}{\hat{\kappa}}

\newcommand{\X}{X}

\newcommand{\HH}{H}

\newcommand{\ka}{\mathfrak{k}}

\newcommand{\xp}{{\rm {exp}}}

\newcommand{\ad}{{ \mbox{\rm Ad}}}

\newcommand{\up}{\Upsilon}

\allowdisplaybreaks

\begin{document}

\title[Drinfeld realisation]{Drinfeld realisations of  quantum affine superalgebras}

\author{Ying Xu}
\author{R. B. Zhang}

\address[Xu]{School of Mathematics, Hefei University of Technology, Anhui Province, 230009, China}
\email{xuying@hfut.edu.cn}
\address[Xu, Zhang]{School of Mathematics and Statistics,
University of Sydney, NSW 2006, Australia}
\email{ruibin.zhang@sydney.edu.au}

\begin{abstract}
We construct Drinfeld realisations for the  quantum affine superalgebras associated with
the $\osp(1|2n)^{(1)}$, $\Sl(1|2n)^{(2)}$ and $\osp(2|2n)^{(2)}$ series of affine Lie superalgebras.
\end{abstract}

\maketitle


\section{Introduction}\label{intro}

The Drinfeld realisation of a quantum affine algebra \cite{Dr}
is a quantum analogue of the loop algebra realisation of an affine Lie algebra.
It is indispensable for studying vertex operator representations \cite{Jn1, JnM} and
finite dimensional representations \cite{CP1, CP2} of the quantum affine algebra.
The equivalence between the Drinfeld realisation and usual Drinfeld-Jimbo presentation in terms of Chevalley generators was known to Drinfeld \cite{Dr}, and has been investigated in a number of papers,  see, e.g.,  \cite{Be,De, Jn2, JZ}.

We construct Drinfeld realisations for a class of quantum affine  superalgebras in this paper.  
Quantum supergroups associated with simple Lie superalgebras and their affine analogues 
were introduced \cite{BGZ, Y94, ZGB91b}  and extensively studied 
(see, e.g., \cite{Z93, Z98} and references therein)  in the 90s. 
They have important applications in a variety of areas such as topology of knots and
$3$-manifolds \cite{Z92a, Z95}, and the theory of integrable models of Yang-Baxter type \cite{BGZ, ZBG91}.
In recent years there has been a resurgence of interest in these algebraic structures. 
Previously Drinfeld realisations were only known for the untwisted quantum affine superalgebras of types $A$ \cite{Y99}
 and $D(2, 1; \alpha)$ \cite{H} in the standard root systems.  The realisation in type $A$  formed the launching pad for the study of integrable representations of the quantum affine special linear superalgebra in \cite{WZ, Zh}.

In this paper, we will focus on the quantum affine superalgebras $\Uq(\g)$ associated with the following series of affine Lie superalgebras $\g$:
\begin{eqnarray}\label{eq:g}
\osp(1|2n)^{(1)},  \quad \Sl(1|2n)^{(2)}, \quad \osp(2|2n)^{(2)}, \quad n\ge 1,
\end{eqnarray}
which are the affine Lie superalgebras that do not have isotropic odd roots.
We construct a Drinfeld realisation $\Dr_q(\g)$ (see Definition \ref{def:DR}) for each
$\Uq(\g)$, and establish a superalgebra isomorphism between $\Dr_q(\g)$ and $\Uq(\g)$
in Theorem \ref{them:DR-iso}. As explained in Remark \ref{rem:hopf-iso},
the isomorphism can in fact be interpreted as an
isomorphism of Hopf superalgebras.

We prove Theorem \ref{them:DR-iso} by relating the Drinfeld realisations of the quantum
affine superalgebras to Drinfeld realisations of some ordinary quantum affine algebras,
and then applying Drinfeld's theorem \cite{Dr}. This makes essential use of
the notion of {\em quantum correspondences} introduced in \cite{XZ}.
A quantum correspondence between a pair $(\g, \g')$ of (affine) Lie superalgebras
is a Hopf superalgebra isomorphism between the corresponding quantum supergroups.
Here we regard the category of vector superspaces as a braided tensor category, 
and a Hopf superalgebra as a Hopf algebra over this category. 
Lemma \ref{lem:correspon} gives a concise description of the quantum correspondences 
used in this paper.

Throughout the paper, $q^{1/2}$ is an indeterminate and $t^{1/2}=\sqrt{-1} q^{1/2}$.
Let $K:=\C(q^{1/2})$ be the field of rational functions in $q^{1/2}$.

\section{ Drinfeld realisations of quantum affine superalgebras}\label{quantum}
Given any affine Lie superalgebra $\g$  in \eqref{eq:g}, we denote by
$A=(a_{ij})$ its  Cartan matrix, and realise $A$ in terms of the set of simple roots
$\Pi=\{\alpha_i \mid i=0, 1, 2, \dots, n\}$
with $a_{i j} = \frac{2(\alpha_i, \alpha_j)}{(\alpha_i, \alpha_i)}$. Let $\tau\subset \{0,1,\dots,n\}$ be the labelling set of the odd simple roots. We recall the following Dynkin diagrams of the affine Lie superalgebras.
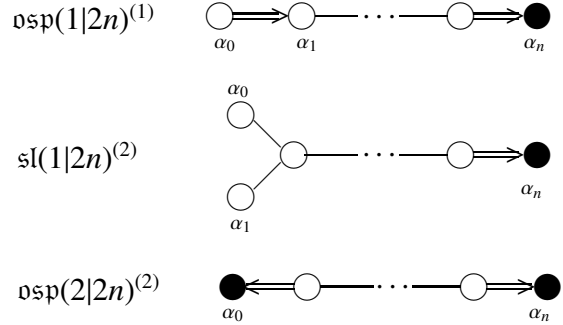
\[
\begin{picture}(120, 28)(40,-12)
\put(-42,-5) {$\osp(1|2n)^{(1)}$}
\put(37,0){\circle{10}}
\put(35,-12){\tiny\mbox{$\alpha_0$}}
\put(42,1){\line(1, 0){17}}
\put(42,-1){\line(1, 0){17}}
\put(56,-3){$>$}
\put(68, 0){\circle{10}}
\put(66,-12){\tiny\mbox{$\alpha_1$}}
\put(73, 0){\line(1, 0){16}}
\put(91, -0.5){\dots}
\put(105, 0){\line(1, 0){18}}
\put(128, 0){\circle{10}}
\put(133,1){\line(1, 0){17}}
\put(133,-1){\line(1, 0){17}}
\put(145,-3){$>$}
\put(157, 0){\circle*{10}}
\put(152,-12){\mbox{\tiny$\alpha_{n}$}}
\end{picture}
\]
\[
\begin{picture}(150, 60)(-20,-28)
\put(-85,-5) {$\Sl(1|2n)^{(2)}$}
\put(0, 15){\circle{10}}
\put(-5,23){\tiny$\alpha_0$}
\put(0, -16){\circle{10}}
\put(-4,-28){\tiny$\alpha_1$}
\put(15, -3){\line(-1, -1){10}}
\put(15, 3){\line(-1, 1){10}}
\put(20, 0){\circle{10}}
\put(24, 0){\line(1, 0){20}}
\put(46, -0.5){\dots}
\put(60,0){\line(1, 0){18}}
\put(83, 0){\circle{10}}
\put(88,1){\line(1, 0){17}}
\put(88,-1){\line(1, 0){17}}
\put(100,-3){$>$}
\put(112, 0){\circle*{10}}
\put(106,-15){\tiny$\alpha_{n}$}
\end{picture}
\]
\[
\begin{picture}(150, 30)(-10,-14)
\put(-75,-5){$\osp(2|2n)^{(2)}$}
\put(7,0){\circle*{10}}
\put(3,-12){\tiny $\alpha_0$}
\put(12,1){\line(1, 0){18}}
\put(12,-1){\line(1, 0){18}}
\put(10,-3){$<$}
\put(35, 0){\circle{10}}
\put(40, 0){\line(1, 0){20}}
\put(61, -0.5){\dots}
\put(75, 0){\line(1, 0){18}}
\put(98, 0){\circle{10}}
\put(103,1){\line(1, 0){17}}
\put(103,-1){\line(1, 0){17}}
\put(114,-3){$>$}
\put(126, 0){\circle*{10}}
\put(116,-12){ \tiny$\alpha_{n}$}
\end{picture}
\]
The black nodes in the Dynkin diagram denote the odd simple roots, while the white ones are even sinple roots. The bilinear form can always be normalised such that  $(\alpha_n,\alpha_n)=1$.

%

Set $q_i=q^{\frac{(\alpha_i,\alpha_i)}{2}}$ for all $\alpha_i\in\Pi$. For any nonzero $z\in K$, denote
\begin{align*}
\begin{bmatrix} N\\k\end{bmatrix}_z=\frac{[N]_z!}{[N-k]_z![k]_z!}, \quad [N]_z!=\prod_{j=1}^N[j]_z,\quad [j]_z=\frac{z^j-z^{-j}}{z-z^{-1}} \ \text{ \ with $[0]_z=1$}.
\end{align*}
For any superalgebra $A=A_{\bar{0}}\bigoplus A_{\bar{1}}$, we define the parity functor $[\,]: A\rightarrow \Z_{2}=\{\bar{0},\bar{1}\}$ on homogeneous elements of $A$ as follows: $[a]=\bar{0}$ if $a\in A_{\bar{0}}$ and $[a]=\bar{1}$  if $a\in A_{\bar{1}}$.

\begin{defi}[\cite{Z2}]\label{defi:quantum-super}
Let $\g=\osp(1|2n)^{(1)}$, $\Sl(1|2n)^{(2)}$ or $\osp(2|2n)^{(2)}$.
The {\em quantum affine  superalgebra} $\Uq(\g)$ over ${K}$  is an associative superalgebra with identity generated by the homogeneous
elements $\qe_i,\qf_i, k_i^{\pm1}$ ($0\le i \le n$), where $\qe_s,\qf_s,(s\in\tau)$, are odd and the other generators are even, with the following defining relations:
\begin{eqnarray}
\nonumber
&& k_i  k_i^{-1}= k_i^{-1}  k_i=1,\quad  k_i  k_j= k_j  k_i,\\
\nonumber
&&     k_i \qe_j  k_i^{-1} = q_i^{a_{ij}}  \qe_j,
\quad  k_i \qf_j   k_i^{-1} = q_i^{-a_{ij}} \qf_j,\\
\label{eq:xx-q}
&&\qe_i\qf_j     -    (-1)^{ [\qe_i][\qf_j] } \qf_j\qe_i
                   =\de_{ij}  \dfrac{  k_i- k_i^{-1} }
                                          { q_i^{\zeta_i}-q_i^{-\zeta_i} }, \quad \forall i, j, \\
\nonumber
&&\left(
            \mbox{Ad}_{\qe_i}    \right)^{1-a_{ij}}    (\qe_j)
    =\left(
             \mbox{Ad}_{\qf_i}    \right)^{1-a_{ij}}    (\qf_j)=0, \quad       \text{ if } i\neq j.
\end{eqnarray}
\end{defi}

Here $\mbox{Ad}_{\qe_i}(x)$ and $\mbox{Ad}_{\qf_i}(x)$ are respectively defined by
\begin{equation}\label{eq:ad}
\begin{aligned}
\mbox{Ad}_{\qe_i}(x)     =         \qe_ix   -(-1)^{[\qe_i][x]}   k_i x  k_i^{-1} \qe_i,\\
\mbox{Ad}_{\qf_i}(x)      =         \qf_ix    -(-1)^{[\qf_i][x]}    k_i^{-1} x  k_i \qf_i,
\end{aligned}
\end{equation}
and $\zeta_i=2$ if $i=n$, and $1$ otherwise.
For any $x, y\in \Uq(\g)$ and $a\in{K}$, we shall write
\[
[x, y]_a=x y - (-1)^{[x][y]} a y x, \quad [x, y] = [x, y]_1.
\]
Then $\mbox{Ad}_{\qe_i}(\qe_j)=[\qe_i, \qe_j]_{q_i^{a_{ij}}}$ and
$\mbox{Ad}_{\qf_i}(\qf_j)=[\qf_i, \qf_j]_{q_i^{a_{ij}}}$.

The superalgebra $\Uq(\g)$ has a Hopf superalgebra structure \cite{Z2}.

\begin{rem}\label{rem:q-1}
Note that the standard definition of the quantum affine superalgebra corresponds to $\zeta_i=1$ for all $i$.
However, Definition \ref{defi:quantum-super} can be transformed to the standard one by an automorphism which,
say, sends $\qe_{i}$ to $[\zeta_i]_{q_i}\qe_i$ for all $i$ and leaves
the other generators intact.
The advantage of Definition \ref{defi:quantum-super} is that  $q^{\pm 1/2}$ never appears in the defining relations.
\end{rem}


To construct the Drinfeld realisation for  the quantum affine superalgebra $\U_q(\g)$, we let $\cI=\{ (i,r)\mid 1\le i\le n,  \ r\in\Z \}$.  Define the set $\cI_\g$ by
$\cI_{\g}:=\cI$
if $\g=\osp(1|2n)^{(1)}$ or $\Sl(1|2n)^{(2)}$; and
$\cI_{\g}:=\cI\backslash \{ (i,2r+1)\mid 1\le i<n, \  r\in \Z\}$
if  $\g=\osp(2|2n)^{(2)}$.
Let $ \mathcal{I}_{\g}^* =\{(i, s)\in \mathcal{I}_{\g}\mid s\ne 0\}$.
Also, for any expression $f(x_{r_1},\dots,x_{r_k})$ in $x_{r_1},\dots,x_{r_k}$, we use $sym_{r_1,\dots,r_k}f(x_{r_1},\dots,x_{r_k})$ to denote $\sum_{\sigma}f(x_{\sigma(r_1)},\dots,x_{\sigma(r_k)})$, where the sum is over the permutation group of the set  $\{r_1, r_2, \dots, r_k\}$.

\begin{defi}\label{def:DR}
For $\g=\osp(1|2n)^{(1)}$, $\Sl(1|2n)^{(2)}$ or $\osp(2|2n)^{(2)}$,  we let $\Dr_q(\g)$ be the associative superalgebra over ${K}$ with identity,  generated by
\[
\ef_{i,r},\ \qk_i^{\pm1}, \  \h_{i,s}, \ \ga^{\pm 1/2}, \quad \text{for }\  (i,r)\in\mathcal{I}_{\g},   \   (i,s)\in\mathcal{I}_{\g}^*, \  1\le i\le n,
\]
where $\e_{n,r},\f_{n,r}$ are odd and the other generators are even, with the following defining relations.
\begin{itemize}
\item[\rm(1)]  $\ga^{\pm 1/2}$ are central, and $\ga^{1/2} \ga^{- 1/2}=1$,
\begin{align}
&\qk_i\qk_i^{-1}=\qk_i^{-1}\qk_i=1,\quad  \qk_i\qk_j=\qk_j\qk_i, \nonumber\\
\label{eq:hx}
& \qk_i \ef_{j,r} \qk_i^{-1}=q_i^{\pm a_{ij}} \ef_{j,r},\quad
 [\h_{i,r},\ef_{j,s}]  = \dfrac{  u_{i,j,r} \ga^{\mp|r|/2}  }
                                                   {   r(q-q^{-1})  }
                                                    \ef_{j,s+r},  \\
  \label{eq:hh}
&[\h_{i,r},\h_{j,s}]=\delta_{r+s,0} \dfrac{ u_{i,j,r} (\ga^{r}-\ga^{-r})  }
                                                           { r   (q-q^{-1})(q-q^{-1})  }  ,\\
\label{eq:xx}
& [\e_{i,r}, \f_{j,s}]  =\delta_{i,j}
                                \dfrac{   \ga^{\frac{r-s}{2}} \hh^{+}_{i,r+s}
                                          -  \ga^{\frac{s-r}{2}} \hh^{-}_{i,r+s}   }
                                         {  q-q^{-1}  },   \\
\label{eq:xrs-xsr}
&[\ef_{i,r\pm \theta}, \ef_{j,s}]_{q^{a_{ij}}_{i}}
  +[\ef_{j,s\pm \theta}, \ef_{i,r}]_{q^{a_{ji}}_{j}}
     =0, \ \
(\g,i,j)\neq (\osp(1|2n)^{(1)}, n, n),
\end{align}
where $\theta=2$ if $\g=\osp(2|2n)^{(2)}, (i,j)\ne (n,n)$,  and  $1$ otherwise; the
$u_{i,j,r}$ are given in \eqref{eq:u-def};  and $\hh^{\pm}_{i,\pm r}$ are defined by
\begin{equation}\label{eq:hh-hat}
\begin{aligned}
&\sum_{r\in\Z}  \hh^{+}_{i,r}u^{-r} =\qk_i \xp  \left(
                                  (q-q^{-1})\sum_{r>0}\h_{i, r}u^{-r}
                                                               \right),\\
&\sum_{r\in\Z}  \hh^{-}_{i,-r}u^r    =\qk_i^{-1}  \xp  \left(
                                 (q^{-1}-q)\sum_{r>0}\h_{i,-r}u^r
                                                                \right);
\end{aligned}
\end{equation}

\item[\rm(2)] {\rm Serre relations}
\begin{itemize}
\item[\rm (A)] $n\ne i\neq j$, \ $\ell=1-a_{i j}$,
\[\begin{aligned}
\hspace{20mm}
sym_{r_1,\dots,r_\ell}\sum_{k=0}^\ell  (-1)^k
                                             \begin{bmatrix} \ell\\k\end{bmatrix}_{q_i}
                                              \ef_{i,r_1}\dots\ef_{i,r_k} \ef_{j,s}\ef_{i,r_k+1}\dots\ef_{i,r_\ell}=0;
\end{aligned}\]

\item[\rm (B)]  $j<n-1$, \ $\ell=1-a_{n j}$, and in case $\g\ne \Sl(1|2n)^{(2)}$, 
\[\begin{aligned}
\hspace{20mm}
sym_{r_1,\dots,r_\ell}\sum_{k=0}^\ell
                                              \begin{bmatrix} \ell\\k\end{bmatrix}_{\sqrt{-1} q_n}
                                              \ef_{n,r_1}\dots\ef_{n,r_k} \ef_{j,s}\ef_{n,r_k+1}\dots\ef_{n,r_\ell}=0;
\end{aligned}\]

\item[\rm (C)] {\rm For} $\g=\osp(1|2n)^{(1)}$,
 \begin{align*}
&sym_{r_1,r_2,r_3} [  [\ef_{n,r_1\pm 1},\ef_{n,r_2}]_{q_n^2},  \ef_{n,r_3}]_{q_n^4}=0;\\
&sym_{r,s}\Big([\ef_{n,r\pm 2}, \ef_{n,s}]_{q_n^2} -q_n^4 [\ef_{n,r\pm 1},  \ef_{n,s\pm 1} ]_{q_n^{-6}}\Big)=0;\\
&sym_{r,s}\Big(q_n^2[[\ef_{n,r\pm1},\ef_{n,s}]_{q_n^2},\ef_{n-1,k}]_{q_n^4}\\
&+(q_n^2+q_n^{-2})[[\ef_{n-1,k},\ef_{n,r\pm1}]_{q_n^2},\ef_{n,s}]\Big)=0;
\end{align*}

\item[\rm (D)] {\rm For} $\g=\osp(2|2n)^{(2)}$,
\[
sym_{r, s} [  [\ef_{n-1 ,k},\ef_{n, r\pm 1}]_{q_n^2},  \ef_{n, s}]=0.
\]
\end{itemize}
\end{itemize}
\end{defi}
In the above, the scalars $u_{i,j,r}$ ($r\in\Z$, $i, j=1, 2, \dots, n$)  are defined by
\begin{eqnarray}\label{eq:u-def}
\begin{aligned}
&\osp(1|2n)^{(1)}: \quad u_{i,j,r}=\begin{cases}
 q_n^{4r}-q_n^{-4r}-q_n^{2r}+q_n^{-2r},   & \text{if } i=j=n,\\
 q_i^{r a_{ij}}- q_i^{-r a_{ij}},                             & \text {otherwise };
            \end{cases}\\
&\osp(2|2n)^{(2)}: \quad u_{i,j,r}=\begin{cases}
  (-1)^r(q_n^{2r}-q_n^{-2r}),                 & \text{if }  i=j=n,  \\
(1+(-1)^r)(q_i^{r a_{ij}/2}-q_i^{-r a_{ij}/2}),        & \text {otherwise };
                  \end{cases}\\
&\Sl(1|2n)^{(2)}: \quad \phantom{X}  u_{i,j,r}=\begin{cases}
(-1)^r(q_n^{2r}-q_n^{-2r}),                   & \text{if }  i=j=n,  \\
q_i^{r a_{ij}}- q_i^{-r a_{ij}},        & \text {otherwise }.
                 \end{cases}
\end{aligned}
\end{eqnarray}

\begin{rem}\label{rem:-q}
Consider the automorphism which leaves the other generators intact but maps
$\h_{i,s}\mapsto \frac{q-q^{-1}}{q_i-q_i^{-1}}\h_{i, s}$ and $\e _{i,s}\mapsto \frac{q-q^{-1}}{q_i-q_i^{-1}} \e_{i, s}$ for all $i$.
It transforms the relations \eqref{eq:hx}-\eqref{eq:hh} into the more standard form
\begin{eqnarray*}
\begin{aligned}
& [\h_{i,r},\ef_{j,s}]  = \dfrac{  u_{i,j,r} \ga^{\mp|r|/2}  }
                                                   {   r(q_i-q_i^{-1})  }
                                                    \ef_{j,s+r},  \\
&[\h_{i,r},\h_{j,s}]=\delta_{r+s,0} \dfrac{ u_{i,j,r} (\ga^{r}-\ga^{-r})    }
                                                           {  r   (q_i-q_i^{-1})(q_j-q_j^{-1})  }           ,\\
& [\e_{i,r}, \f_{j,s}]  =\delta_{i,j}
                                \dfrac{   \ga^{\frac{r-s}{2}}  \hh^{+}_{i,r+s}
                                          -  \ga^{\frac{s-r}{2}} \hh^{-}_{i,r+s}   }
                                         {  q_i-q_i^{-1}  }, \\
&\sum_{r\in\Z}  \hh^{+}_{i,r}u^{-r} =\qk_i \xp  \left(
                                  (q_i-q_i^{-1})\sum_{r>0}\h_{i, r}u^{-r}
                                                               \right),\\
&\sum_{r\in\Z}  \hh^{-}_{i,-r}u^r    =\qk_i^{-1}  \xp  \left(
                                 (q_i^{-1}-q_i)\sum_{r>0}\h_{i,-r}u^r
                                                                \right).
\end{aligned}
\end{eqnarray*}
\end{rem}

The following theorem is the main result of this paper; its proof 
will be given in the next section.
\begin{theo}[Main Theorem]\label{them:DR-iso}
Let $\g=\osp(1|2n)^{(1)}$, $\Sl(1|2n)^{(2)}$ or $\osp(2|2n)^{(2)}$.
There exists a superalgebra isomorphism
$
\Psi: \Uq(\g)\stackrel{\sim}\longrightarrow \Dr_q(\g)
$
such that
\begin{align*}
\text {for } \g=&\osp(1|2n)^{(1)}:\\
&\qe_i \mapsto \e_{i,0}, \quad \qf_i\mapsto \f_{i,0}, \quad k_i \mapsto \qk_i, \quad  k^{-}_i\mapsto \qk^{-}_i,\quad \text {for } 1\le i\le n,\\
& \qe_0\mapsto\ad_{\f_{1,0}} \dots \ad_{\f_{n,0}}\ad_{\f_{n,0}}\ad_{\f_{n-1,0}} \dots \ad_{\f_{2,0}}(\f_{1,1}) \ga \qk^{-1}_{\g},\\
& \qf_0\mapsto c_{\g}\ga^{-1}\qk_{\g}\ad_{\e_{1,0}} \dots \ad_{\e_{n,0}}\ad_{\e_{n-1,0}} \dots \ad_{\e_{2,0}} (\e_{1,-1}),\quad  k_0\mapsto\ga \qk_{\g}^{-1},\\
\text {for } \g=& \Sl(1|2n)^{(2)}:\\
&\qe_i \mapsto \e_{i,0}, \quad \qf_i\mapsto \f_{i,0}, \quad k_i \mapsto \qk_i, \quad  k^{-}_i\mapsto \qk^{-}_i,\quad \text {for } 1\le i\le n,\\
& \qe_0\mapsto\ad_{\f_{2,0}} \dots \ad_{\f_{n,0}}\ad_{\f_{n,0}}\ad_{\f_{n-1,0}} \dots \ad_{\f_{2,0}}(\f_{1,1}) \ga \qk^{-1}_{\g},\\
& \qf_0\mapsto c_{\g}\ga^{-1}\qk_{\g}\ad_{\e_{2,0}} \dots \ad_{\e_{n,0}}\ad_{\e_{n-1,0}} \dots \ad_{\e_{2,0}} (\e_{1,-1}),\quad  k_0\mapsto\ga \qk_{\g}^{-1},\\
\text {for } \g=&\osp(2|2n)^{(2)}:\\
&\qe_i \mapsto \e_{i,0}, \quad \qf_i\mapsto \f_{i,0}, \quad k_i \mapsto \qk_i, \quad  k^{-}_i\mapsto \qk^{-}_i,\quad \text {for } 1\le i\le n,\\
&\qe_0\mapsto\ad_{\f_{1,0}} \dots \ad_{\f_{n-1,0}}(\f_{n,1}) \ga \qk^{-1}_{\g},\\
&\qf_0\mapsto c_{\g}\ga^{-1}\qk_{\g}\ad_{\e_{1,0}} \dots \ad_{\e_{n-1,0}}(\e_{n,-1}),
\quad k_0\mapsto\ga \qk_{\g}^{-1},
\end{align*}
where $k_{\g}$ is defined by
\begin{align*}
&\qk_{\g}=\begin{cases}
\qk_1^2\qk_2^2\dots \qk_{n}^2, & \g=\osp(1|2n)^{(1)},\\
\qk_1\qk_2^2\dots \qk_{n}^2,     &\g=\Sl(1|2n)^{(2)},\\
\qk_1\qk_2\dots \qk_n,                & \g=\osp(2|2n)^{(2)},
\end{cases}
\end{align*}
and $c_{\g}\in{K}$  is determined by \eqref{eq:xx-q}.
\end{theo}
\begin{rem}\label{rem:hopf-iso}
We can transcribe the Hopf superalgebra structure of $\Uq(\g)$ to
$\Dr_q(\g)$ using $\Psi$. For example, if $\Delta$ is the 
co-multiplication of $\Uq(\g)$, the co-multiplication of $\Dr_q(\g)$ is 
given by $(\Psi\otimes\Psi)\circ\Delta\circ\Psi^{-1}$. 
Then clearly $\Psi$ is an isomorphism of Hopf superalgebras.
\end{rem}

\section{Proof of the main theorem}\label{sec:DR}
We prove Theorem \ref{them:DR-iso} in this section.
\subsection{Smash products}\label{sect:smash-prod}
Let $\g$ be any of the affine Lie superalgebras
in the first row of Table \ref{tbl:g} (which is \eqref{eq:g}),
\vspace{-2mm}
\begin{table}[h]
\caption{Table 1} \label{tbl:g}
\renewcommand{\arraystretch}{1.2}
\vspace{-2mm}
\begin{tabular}{c|c|c|c}
\hline
$\g$ & $\osp(1|2n)^{(1)}$    & $\Sl(1|2n)^{(2)}$ & $\osp(2|2n)^{(2)}$  \\
\hline
$\g'$ & $A_{2n}^{(2)}$   & $B_n^{(1)}$   & $D_{n+1}^{(2)}$   \\
\hline
\end{tabular}
\end{table}
and let $\g'$ be the ordinary affine Lie algebra corresponding to $\g$
in the second row.
We will speak about the pair $(\g, \g')$ of affine Lie (super)algebras in the table.
Now $\g'$ has the same
Cartan matrix  $A=(a_{ij})$ as $\g$. We let $\Pi'=\{\alpha'_0, \alpha'_1, \dots, \alpha'_n\}$ be the set of simple roots of $\g'$ which realises the Cartan matrix, and take $(\alpha'_n, \alpha'_n)=1$.  Recall $t^{1/2}=\sqrt{-1} q^{1/2}$ and let $t_i=t^{(\alpha'_i,\alpha'_i)/2}$ for all $i$.

The quantum affine  algebra  $\U_{t}(\g')$ is an associative  algebra over ${K}$ with identity generated by the elements $\qe'_i,\qf'_i, {k'_i}^{\pm1}$ ($0\le i \le n$) with the following defining relations:
\begin{eqnarray}
\nonumber
&& k'_i  {k'_i}^{-1}= {k'_i}^{-1}  k'_i=1,\quad  k'_i  k'_j= k'_j  k'_i,\\
\nonumber
&&     k'_i \qe'_j  {k'_i}^{-1} = t_i^{a_{ij}}  \qe'_j,
\quad  k'_i \qf'_j  {k'_i}^{-1}= t_i^{-a_{ij} }  \qf'_j,\\
\label{eq:ef-A}
&&\qe'_i\qf'_j     -   \qf'_j\qe'_i
                   =\de_{i,j} \dfrac{  k'_i- {k'_i}^{-1} }
                                          { t_i^{\zeta_i}-t_i^{-\zeta_i}}, \quad \forall i, j; \\
\nonumber
&&\left(
            \mbox{Ad}_{\qe'_i}    \right)^{1-a_{ij}}    (\qe'_j)
    =\left(
             \mbox{Ad}_{\qf'_i}    \right)^{1-a_{ij}}    (\qf'_j)=0, \quad       \text{ if } i\neq j.
\end{eqnarray}
Here $\zeta_i=2$ if $i=n$, and $1$ otherwise; $\mbox{Ad}_{\qe'_i}(x)$ and $\mbox{Ad}_{\qf'_i}(x)$ are respectively defined by
\begin{equation*}
\begin{aligned}
&\mbox{Ad}_{\qe'_i}(x)     =         \qe'_ix   -     k'_i x  {k'_i}^{-1} \qe'_i,\\
&\mbox{Ad}_{\qf'_i}(x)      =         \qf'_ix    -    {k'_i}^{-1} x  k'_i \qf'_i.
\end{aligned}
\end{equation*}
It is well known that $\Uq(\g')$ is a Hopf algebra.

\begin{rem}\label{rem:q-2}
The definition of $\Uq(\g')$ given here is related to the standard one by an automorphism analogous to that given in Remark \ref{rem:q-1}.
\end{rem}

Let $\mathcal{I}_{\g'}=\mathcal{I}_{\g}$ and $\mathcal{I}_{\g'}^*=\mathcal{I}_{\g}^*$
($\mathcal{I}_{\g}$ and $\mathcal{I}_{\g}^*$ are defined before
Definition \ref{def:DR}). The Drinfeld realisation $\Dr_{t}(\g')$ of $\U_t(\g')$  is an associative algebra over ${K}$ with identity generated by the generators
\[
\x^{\pm}_{i,r},\ \qk'^{\pm1}_i, \  \h'_{i,s}, \ \ga'^{\pm 1/2}, \quad \text{for }\  (i,r)\in\mathcal{I}_{\g'},   \   (i,s)\in\mathcal{I}_{\g'}^*, \  1\le i\le n,
\]
with the defining relations \cite{Dr}:

\begin{enumerate}
\item[(1)] $[{\ga'}^{\pm 1/2},\x^{\pm}_{i,r}]  =  [{\ga'}^{\pm 1/2},\h'_{i,r}]=[{\ga'}^{\pm 1/2},\qk'_{i}]=0$,
\begin{eqnarray}
\nonumber
&& \qk'_i {\qk'_i}^{-1}  = {\qk'_i}^{-1} \qk'_i=1, \quad \qk'_i \qk'_j=\qk'_j \qk'_i,\\
\nonumber
&& \qk'_i \x^{\pm}_{j,r} {\qk'_i}^{-1}  = t_i^{\pm a_{ij}} \x^{\pm}_{j,r},\quad
[\h'_{i,r},\x^{\pm}_{j,s}] = \dfrac
                                               {u'_{i,j,r}}
                                               {r(t -t^{-1})}
                                        {\ga'}^{\mp|r|/2}  \x^{\pm}_{j,s+r},\\
\nonumber
&&[\h'_{i,r},\h'_{j,s}] = \delta_{r+s,0} \dfrac
                                                         {   u'_{i,j,r}
                                                               (  \ga'^{r}- \ga'^{-r}  )   }
                                                        {      r(t -t ^{-1})(t -t ^{-1})    },\\
\nonumber
&&[\x^{-}_{i,r}, \x^{+}_{j,s}] = \de_{ij}   \dfrac
                                                           {    \ga'^{\frac{r-s}{2}}           \hh'^{+}_{i,r+s}
                                                             -  \ga'^{\frac{s-r}{2}}   \hh'^{-}_{i,r+s}    }
                                                           {q^{-1}-q },\\
\label{eq:sym-dr}
&&sym_{r,s}[\x^{\pm}_{i,r\pm \theta}, \x^{\pm}_{j,s}]_{t^{a_{ij}}_{i}}
     =0, \ \ \text{for \ } (\g,i,j)\neq (A_{2n}^{(2)}, n, n),
\end{eqnarray}
where $\theta=2$ if $\g=D_{n+1}^{(2)}$,  $(i,j)\ne (n,n)$,  and $1$ otherwise;  $\hh'^{\pm}_{i,\pm r}$ are defined by
\begin{equation}\label{eq:hh'}
\begin{aligned}
&\sum_{r\in\Z}  \hh'^{+}_{i,r}u^{-r}  =  \qk'_i  \xp\left(
                                         (t-t^{-1})  \sum_{r>0} \h'_{i, r} u^{-r}
                                                               \right),\\
&\sum_{r\in\Z}  \hh'^{-}_{i,-r}u^r  = \qk'^{-1}_i  \xp\left(
                                        (t^{-1}-t)   \sum_{r>0}  \h'_{i,-r} u^r
                                                            \right),
\end{aligned}
\end{equation}
and the scalars $u'_{i,j,r}$ are given in \eqref{eq:u'};

\item[\rm(2)] {\rm Serre relations}
\begin{itemize}
\item[(A)] $n\ne i\neq j$, or $\g'\neq D_{n+1}^{(2)}$, $j+1<i=n$,\ $\ell=1-a_{i j}$ ,
\[\begin{aligned}
sym_{r_1,\dots,r_\ell}\sum_{k=0}^\ell  (-1)^k
      \begin{bmatrix} \ell\\k\end{bmatrix}_{t}
     \x^{\pm}_{i,r_1}\dots\x^{\pm}_{i,r_k}\x^{\pm}_{j,s}\x^{\pm}_{i,r_k+1}
                              \dots\x^{\pm}_{i,r_\ell}=0;
\end{aligned}\]

\item[\rm (B)] {\rm For} $\g=A_{2n}^{(2)}$,
 \begin{align*}
&sym_{r_1,r_2,r_3} [  [\ef_{n,r_1\pm 1},\ef_{n,r_2}]_{t_n^2},  \ef_{n,r_3}]_{t_n^4}=0;\\
&sym_{r,s}\Big([\ef_{n,r\pm 2}, \ef_{n,s}]_{t_n^2} -t_n^4 [\ef_{n,r\pm 1},  \ef_{n,s\pm 1} ]_{t_n^{-6}}\Big)=0;\\
&sym_{r,s}\Big(t_n^2[[\ef_{n,r\pm1},\ef_{n,s}]_{t_n^2},\ef_{n-1,k}]_{t_n^4}\\
&+(t_n^2+t_n^{-2})[[\ef_{n-1,k},\ef_{n,r\pm1}]_{t_n^2},\ef_{n,s}]\Big)=0;
\end{align*}

\item[\rm (C)] {\rm For} $\g=D_{n+1}^{(2)}$,
\[
sym_{r, s} [  [\ef_{n-1 ,k},\ef_{n, r\pm 1}]_{t_n^2},  \ef_{n, s}]=0.
\]
\end{itemize}
\end{enumerate}
The scalars $u'_{i,j,r}$ in the above equations are defined by
\begin{equation}\label{eq:u'}
\begin{aligned}
&A_{2n}^{(2)}: \quad u'_{i,j,r}=\begin{cases}
( t_n^{2r}-t_n^{-2r})(t_n^{2r}+t_n^{-2r}+(-1)^{r-1}),   & \text{if } i=j=n,\\
 t_i^{r a_{ij}}- t_i^{-r a_{ij}},                                        & \text {otherwise };
            \end{cases}\\
&B_n^{(1)}: \quad \phantom{X}  u'_{i,j,r}=t_i^{r a_{ij}}- t_i^{-r a_{ij}};\\
&D_{n+1}^{(2)}: \quad u'_{i,j,r}=\begin{cases}
t_n^{2r}-t_n^{-2r},                 & \text{if }  i=j=n,  \\
(1+(-1)^r)(t_i^{r a_{ij}/2}-t_i^{-r a_{ij}/2}),        & \text {otherwise }.
                  \end{cases}
\end{aligned}
\end{equation}

\begin{rem}\label{rem:q-3}
The above is the Drinfeld realisation \cite{Dr} of the quantum affine algebra $\U_t(\g')$
for each $\g'$, taking into account the variation of the definition of $\U_t(\g')$ discussed in Remark \ref{rem:q-3}.
\end{rem}



Applied to the quantum affine algebras under consideration, Drinfeld's  theorem \cite{Dr}
gives the following algebra isomorphism
\begin{eqnarray}\label{eq:JD-Dr-alg}
\rho:\U_{t}(\g')\stackrel{\sim}{\longrightarrow}\Dr_{t}(\g');
\end{eqnarray}
\begin{equation}\label{eq:iso-alg}
\begin{aligned}
\text {for } \g'=&A_{2n}^{(2)}:\\
&\qe'_i \mapsto \x^{+}_{i,0}, \quad \qf'_i\mapsto \x^{-}_{i,0}, \quad k'_i \mapsto \qk'_i, \quad  k'^{-}_i\mapsto \qk'^{-}_i,\quad \text {for } 1\le i\le n,\\
& \qe'_0\mapsto\ad_{\x^{-}_{1,0}} \dots \ad_{\x^{-}_{n,0}}\ad_{\x^{-}_{n,0}}\ad_{\x^{-}_{n-1,0}} \dots \ad_{\x^{-}_{2,0}}(\x^{-}_{1,1}) \ga' \qk'^{-1}_{\g'},\\
& \qf'_0\mapsto c_{\g'}\ga'^{-1}\qk'_{\g'}\ad_{\x^{+}_{1,0}} \dots \ad_{\x^{+}_{n,0}}\ad_{\x^{+}_{n-1,0}} \dots \ad_{\x^{+}_{2,0}} (\x^{+}_{1,-1}),\quad  k'_0\mapsto\ga' \qk'^{-1}_{\g'},\\
\text {for } \g'=& B_n^{(1)}:\\
&\qe'_i \mapsto \x^{+}_{i,0}, \quad \qf'_i\mapsto \x^{-}_{i,0}, \quad k'_i \mapsto \qk'_i, \quad  k'^{-}_i\mapsto \qk'^{-}_i,\quad \text {for } 1\le i\le n,\\
& \qe'_0\mapsto\ad_{\x^{-}_{2,0}} \dots \ad_{\x^{-}_{n,0}}\ad_{\x^{-}_{n,0}}\ad_{\x^{-}_{n-1,0}} \dots \ad_{\x^{-}_{2,0}}(\x^{-}_{1,1}) \ga' \qk'^{-1}_{\g'},\\
& \qf'_0\mapsto c_{\g'}\ga'^{-1}\qk'_{\g'}\ad_{\x^{+}_{2,0}} \dots \ad_{\x^{+}_{n,0}}\ad_{\x^{+}_{n-1,0}} \dots \ad_{\x^{+}_{2,0}} (\x^{+}_{1,-1}),\quad  k'_0\mapsto\ga' \qk'^{-1}_{\g'},\\
\text {for } \g'=&D_{n+1}^{(2)}:\\
&\qe'_i \mapsto \x^{+}_{i,0}, \quad \qf'_i\mapsto \x^{-}_{i,0}, \quad k'_i \mapsto \qk'_i, \quad  k'^{-}_i\mapsto \qk'^{-}_i,\quad \text {for } 1\le i\le n,\\
&\qe'_0\mapsto\ad_{\x^{-}_{1,0}} \dots \ad_{\x^{-}_{n-1,0}}(\x^{-}_{n,1}) \ga' \qk'^{-1}_{\g'},\\
&\qf'_0\mapsto c_{\g'}\ga'^{-1}\qk'_{\g'}\ad_{\x^{+}_{1,0}} \dots \ad_{\x^{+}_{n-1,0}}(\x^{+}_{n,-1}),
\quad k'_0\mapsto\ga' \qk'^{-1}_{\g'},
\end{aligned}
\end{equation}
where $\qk'_{\g'}$ is defined by
\begin{align*}
&\qk'_{\g}=\begin{cases}
\qk'^2_1\qk'^2_2\dots \qk'^2_{n}, & \g'=A_{2n}^{(2)},\\
\qk'_1\qk'^2_2\dots \qk'^2_{n},     &\g'=B_{n}^{(1)},\\
\qk'_1\qk'_2\dots \qk'_n,                & \g'=D_{n+1}^{(2)};
\end{cases}
\end{align*}
and $c_{\g}\in{K}$  can be fixed by \eqref{eq:ef-A}.

To prove Theorem \ref{them:DR-iso}, we will need to enlarge the quantum affine superalgebra $\Uq(\g)$
and the Drinfeld superalgebra $\Dr_q(\g)$  following \cite{XZ,Z2}.

Corresponding to each simple root  $\alpha_i$ of $\g$ for $i\neq 0$, we introduce a  group $\Z_2$ generated by $\sigma_i$ such that ${\sigma_i}^2=1$, and let $\mathrm{G}$ be the direct product of all such groups.
The group algebra ${K}\mathrm{G}$ has a standard Hopf algebra structure with the
co-multiplication given by $\Delta(\sigma_i)=\sigma_i\otimes\sigma_i$ for all $i$.
Define a left $\mathrm{G}$-action on $\Uq(\g)$ by
\begin{align}
\sigma_i\cdot e_j=(-1)^{(\alpha_i,\alpha_j)}e_j, \quad \sigma_i\cdot f_j=(-1)^{(\alpha_i,\alpha_j)}f_j, \quad \sigma_i\cdot k_j=k_j, \quad\text{$i\ne 0$},
\end{align}
which preserves the multiplication of $\Uq(\g)$.  This  defines a  left ${K}\mathrm{G}$-module algebra structure on $\Uq(\g)$. Similarly, let $\mathrm{G}$ act on $\Dr_q(\g)$ by
\begin{align}\label{eq:G-act}
&\sigma_i\cdot\xi^\pm_{j, r}=(-1)^{(\alpha_i,\alpha_j)}\xi^\pm_{j, r},
\quad \sigma_i\cdot\kappa_{j, r}= \kappa_{j, r},
\quad  \sigma_i\cdot\gamma_j=\gamma_j,  \quad  \sigma_i\cdot\gamma=\gamma,
\end{align}
for all $i, j\ge 1$ and $r\in\Z$. This again preserves the multiplication of  $\Dr_q(\g)$.

By using a standard construction in the theory of Hopf algebras,  we manufacture the smash product superalgebras
\begin{eqnarray}
\UU_q(\g):=\Uq(\g)\sharp{K}\mathrm{G}, \quad \D_q(\g):=\Dr_q(\g)\sharp{K}\mathrm{G},
\end{eqnarray}
which have underlying vector superspaces  $\Uq(\g)\otimes{K}\mathrm{G}$
and $\Dr_q(\g)\otimes{K}\mathrm{G}$ respectively,
where ${K}\mathrm{G}$ is regarded as purely even.
The multiplication of  $\UU_q(\g)$ (resp. $\D_q(\g)$) is defined,
for all $x, y$ in  $\Uq(\g)$ (resp. $\D_q(\g)$)  and $\sigma, \tau\in \mathrm{G}$, by
\[
(x\otimes \sigma)(y\otimes \tau)= x \sigma.y\otimes \sigma\tau.
\]
We will write $x \sigma$ and $\sigma x$ for $x\otimes \sigma$ and  $(1\otimes\sigma)(x\otimes 1)$ respectively.

In exactly the same way, we introduce a  group $\Z_2$ corresponding to each simple root  $\alpha'_i$ of $\g'$ with $i\neq 0$. The group is generated by $\sigma'_i$ such that ${\sigma'_i}^2=1$.
Let $\mathrm{G}'$ be the direct product of all such groups, and define a $\mathrm{G}'$-action on $\U_{t}(\g')$  by
\begin{align}
&\sigma'_i\cdot e'_j=(-1)^{(\alpha'_i,\alpha'_j)}e'_j, \quad \sigma'_i\cdot f'_j=(-1)^{-(\alpha'_i,\alpha'_j)}f'_j, \quad \sigma'_i\cdot k'_j=k'_j, \quad\text{$i\ne 0$}.
\end{align}
This induces a $\mathrm{G}'$-action on $\Dr_t(\g')$ analogous to \eqref{eq:G-act}.
Now we introduce the smash product algebras
\[
\UU_{t}(\g')=\U_{t}(\g')\sharp{K}\mathrm{G}', \quad \D_{t}(\g')=\Dr_{t}(\g')\sharp{K}\mathrm{G}'.
\]
Clearly we can extend equation \eqref{eq:JD-Dr-alg} to the algebra isomorphism
\begin{eqnarray}\label{eq:dr-iso-A}
\UU_{t}(\g')\stackrel{\sim}{\longrightarrow}\D_{t}(\g'),
\end{eqnarray}
which is the identity on ${K}\mathrm{G}'$.

\subsection{Quantum correspondences}\label{sect:QCs}

We classified the quantum correspondences in \cite[Theorem 4.9]{XZ}; 
the following ones are relevant to the present paper,
which were first established in \cite{Z2}.
\begin{lem}[\cite{XZ}, \cite{Z2}]\label{lem:correspon}
For each pair $(\g, \g')$ in Table \ref{tbl:g}, there exists an isomorphism
$
\psi: \UU_{q}(\g)\stackrel{\sim}{\longrightarrow}\UU_{t}(\g')
$
of associative algebras given by
\begin{equation}\label{eq:affine-map}
\begin{aligned}
&e_0 \mapsto \iota_{e}e'_0, \quad  \ f_0\mapsto \iota_{f} f'_0, \quad  \ k^{\pm 1}_0 \mapsto  \iota_{e} \iota_{f} k'^{\pm 1}_0,\\
& \sigma_i\mapsto \sigma'_i,\quad
e_i \mapsto \left(\prod_{k=i+1}^{m+n}\sigma'_k\right)  e'_i, \quad f_i \mapsto \left(\prod_{k=i}^{m+n}\sigma'_k\right) f'_i, \quad k_i \mapsto \sigma'_i  k'_i,
\end{aligned}
\end{equation}
for $i\neq 0$, where  $\iota_{e}, \iota_{f}\in{K}G'$ are defined by
\begin{equation*}
    \iota_{e}=\begin{cases}
1, & \g'=A_{2n}^{(2)},\\
\prod_{i=2}^n\sigma'_i, &\g'=B_{n}^{(1)},\\
\prod_{i=0}^n\sigma'_{2+2i} ,&\g'=D_{n+1}^{(2)},
           \end{cases}
\quad
  \iota_{f}=\begin{cases}
1, & \g'=A_{2n}^{(2)},\\
\prod_{i=1}^n\sigma'_i  , &\g'=B_{n}^{(1)},\\
\prod_{i=0}^n\sigma'_{1+2i},&\g'=D_{n+1}^{(2)}
           \end{cases}
\end{equation*}
with $\sigma'_j=1$ for all $j>n$.
\end{lem}

\begin{rem} 
Within the context of Hopf algebras over braided tensor categories, 
the above associative algebra isomorphism becomes a Hopf algebra 
isomorphism; see \cite[Section 4]{XZ} for details.
\end{rem}

Now for each $\g'$, we introduce a  surjection $o: I_n=\{1,\dots,n\}\to \{\pm 1\}$ defined by  $o(i)=(-1)^{n-i}$. We also define $c:=c(\g)$  such that
$c=1/2$ if $\g=\osp(2|2n)^{(2)}$,  and  1 otherwise.

We have the following result.
\begin{theo}\label{lem:iso}
For each pair $(\g, \g')$ in Table \ref{tbl:g},
there is an isomorphism $\varphi: \D_{q}(\g)\stackrel{\sim}{\longrightarrow}\D_{t}(\g')$
of associative algebras given by
\begin{equation}\label{eq:dr-map}
\begin{aligned}
& \ga\mapsto\ga', \ \ \h_{i,r} \mapsto -o(i)^{r c}  \h'_{i,r},
       \ \ \qk^{\pm 1}_i \mapsto \sigma'_i  \qk'^{\pm 1}_i,
       \ \ \sigma_i\mapsto \sigma'_i,\\
&     \e_{i,r} \mapsto o(i)^{r c} \left(\prod_{k=i+1}^{m+n}\sigma'_k\right)  \x^{+}_{i,r},
\quad \f_{i,r}  \mapsto o(i)^{r c} \left(\prod_{k=i+1}^{m+n}\sigma'_k\right)  \x^{-}_{i,r}.
\end{aligned}
\end{equation}
\end{theo}
\begin{proof}
If we can show that the map $\varphi$ indeed gives rise to a homomorphism of associative algebras,  then by inspecting \eqref{eq:dr-map}, we immediately see that it is an isomorphism with the inverse map given by
\begin{equation}\label{eq:iso-inv}
\begin{aligned}
\varphi^{-1}:\quad
&  \ga'\mapsto\ga, \quad \h'_{i,r} \mapsto -o(i)^{c r}  \h_{i,r},
        \quad \qk'_i \mapsto \sigma_i  \qk_i,
        \quad \sigma'_i\mapsto \sigma_i,\\
&\      \x^{+}_{i,r} \mapsto o(i)^{c r} \left(\prod_{k=i+1}^{m+n}\sigma_k\right)  \e_{i,r},
\quad \x^{-}_{i,r}  \mapsto o(i)^{c r} \left(\prod_{k=i+1}^{m+n}\sigma_k\right)  \f_{i,r}.
\end{aligned}
\end{equation}
We prove that $\varphi$ is an algebra homomorphism by showing that
the elements $\varphi(\ef_{i,r})$, $\varphi(\h_{i,r})$, $\varphi(\qk^{\pm}_i)$, $\varphi(\ga)$, $\varphi(\sigma_i)$ in $\D_q(\g') $ satisfy the  defining relations of $\D_q(\g)$.

Let us start by verifying the first set of relations in Definition \ref{def:DR}.
Using $\x^{\pm}_{i,r}\sigma'_j=(-1)^{(\alpha'_i,\alpha'_j)}\sigma'_j\x^{\pm}_{i,r}$ and $(-1)^{(\alpha'_i,\alpha'_j)}t_i^{\pm a_{ij}}=q_i^{\pm a_{ij}}$, we immediately obtain
\[
\varphi(\qk_i) \varphi(\ef_{j,r}) \varphi(\qk_i^{-1})=q_i^{\pm a_{ij}} \varphi(\ef_{j,r}).
\]
 Since $u_{i,j,r}=o(i)^{c r}o(j)^{c r}u'_{i,j,r}$,  we have
\begin{eqnarray*}
&[\varphi(\h_{i,r}), \varphi(\ef_{j,s})]  = \dfrac{  u_{i,j,r} \varphi(\ga)^{\mp|r|/2}  }
                                                   {   r(q-q^{-1})  }
                                                   \varphi( \ef_{j,s+r}),\\
&[\varphi(\h_{i,r}),\varphi(\h_{j,s})]=\delta_{r+s,0}
                              \dfrac{ u_{i,j,r} (\varphi(\ga)^{r}-\varphi(\ga)^{-r})    }
                                                           {  r   (q-q^{-1})(q-q^{-1})  } .
\end{eqnarray*}
Let $\Phi'_{j}=\prod_{k=j}^n\sigma'_k$.  Then
\begin{eqnarray*}
\x^{+}_{n,r}\Phi'_{j}=(-1)^{\delta_{n,j}}\Phi'_{j}\x^{+}_{n,r},
&&\x^{+}_{i,r}\Phi'_{j}=(-1)^{\delta_{i,j}+\delta_{i+1,j}}\Phi'_{j}\x^{+}_{i,r}, \quad i\neq n.
\end{eqnarray*}
Using this we obtain
\begin{eqnarray*}
&&\varphi(\e_{i,r})\varphi(\f_{j,s})-(-1)^{[\e_{i,r}][\f_{j,s}]}\varphi(\f_{j,s})\varphi(\e_{i,r}) \\
&&=\de_{i,j} \dfrac{   \varphi(\ga)^{\frac{r-s}{2}} \varphi(\qk_i)\varphi(\hh^{+}_{i,r+s})
                                   -  \varphi(\ga)^{\frac{s-r}{2}}\varphi(\qk_i)^{-1}\varphi(\hh^{-}_{i,r+s})   }
                                         {  q-q^{-1}  },
\end{eqnarray*}
where we have used $\varphi(\hh^{+}_{i,r+s})=o(i)^{c  r}\hh'^{+}_{i,r+s}$ since $\varphi(\h_{i,r})= -o(i)^{c  r} \h'_{i,r}$. Now we have the obvious relations
\begin{eqnarray*}
&&[\varphi(\e_{n,r}), \varphi(\e_{j,s})]_{q^{a_{nj}}_{n}}=(-1)^{\delta_{n-1,j}}\Phi'_{j+1}
     [\x^{+}_{n,r},\x^{+}_{j,s}]_{t^{a_{nj}}_{n}},\\
&&[\varphi(\e_{i,r}), \varphi(\e_{j,s})]_{q^{a_{ij}}_{i}}
                                                      =(-1)^{\delta_{i,j}+\delta_{i-1,j}}\Phi'_{i+1}\Phi'_{j+1}
    [\x^{+}_{i,r},\x^{+}_{j,s}]_{t^{a_{ij}}_{i}}, i\neq n.
\end{eqnarray*}
Using them  together with \eqref{eq:sym-dr}, we obatin
$sym_{r,s}[\varphi(\e_{i,r+ \theta}),\varphi(\e_{j,s})]_{q^{a_{ij}}_{i}}=0$,  if $(\g,i,j)\neq (A_{2n}^{(2)}, n, n)$.  We can similarly show that $sym_{r,s}[\varphi(\f_{i,r- \theta}),\varphi(\f_{j,s})]_{q^{a_{ij}}_{i}}=0$.

The Serre relations in Definition \ref{def:DR}
can be verified in the same way.  For example, in the case $\g'=B_n^{(1)}$, we have
\[\begin{aligned}
&    sym_{r_1,r_2,r_3} \sum_{k=0}^3\begin{bmatrix} 3\\k\end{bmatrix}_{\sqrt{-1} q_{n}}
               {\varphi(\e_{n,r_1})}\dots
                          {\varphi(\e_{n,r_k})} \varphi(\e_{n-1,s}){\varphi(\e_{n,r_{k+1}})}
                  \dots {\varphi(\e_{n,r_3})}\\
&=\sigma'_n sym_{r_1,r_2,r_3}
     \sum_{k=0}^3(-1)^k\begin{bmatrix} 3\\k\end{bmatrix}_{t_n}
              {\x^{+}_{n,r_1}}\dots
                      {\x^{+}_{n,r_k}}\x^{+}_{n-1,s}{\x^{+}_{n,r_{k+1}}}
               \dots {\x^{+}_{n,r_3}}=0.
\end{aligned}
\]
We omit the proof of the other Serre relations.
\end{proof}

\subsection{Proof of Theorem \ref{them:DR-iso}}\label{sec:DR-osp1}
The Main Theorem is an easy consequence of Theorem \ref{lem:iso}.
\begin{coro}
Theorem \ref{them:DR-iso} holds for each pair $(\g, \g')$ in Table \ref{tbl:g}.
\end{coro}
\begin{proof} By composing the isomorphism \eqref{eq:JD-Dr-alg} with those in Lemma \ref{lem:correspon} and Theorem  \ref{lem:iso},
we immediately  obtain the algebra isomorphism
$$\Phi=\varphi\circ\rho\circ\psi: \UU_{q}(\g)  \mapsto  \D_{q}(\g).$$
Note that $\Phi$ preserves the $\Z_2$-grading, thus is an isomorphism of superalgebras.

One can easily check that $\Phi(\U_q(\g)\otimes 1)=\Dr_q (\g)\otimes 1$. Let
$\eta: \U_{q}(\g) \longrightarrow  \UU_{q}(\g)$ be the embedding $x\mapsto x\otimes 1$,
and $\upsilon: \Dr_q (\g)\otimes 1 \longrightarrow \Dr_q (\g)$ be the natural isomorphism
$y\otimes 1\mapsto y$.  Then $\upsilon\circ\Phi\circ\eta$ is the superalgebra isomorphism $\Psi$ of Theorem \ref{them:DR-iso}.
\end{proof}

We comment on a possible alternative approach to the proof of Theorem \ref{them:DR-iso}.
For the affine Lie superalgebras in \eqref{eq:g},
the combinatorics of the affine Weyl groups of the
root systems essentially controls the structures of the affine Lie superalgebras themselves.
The corresponding quantum affine superalgebras
have enough Lusztig automorphisms, which can in principle be used to prove
Theorem \ref{them:DR-iso} by following the approach of \cite{Be}.
It will be interesting to work out the details of such a proof,
although it is expected to be much more involved than the one given here.

\section*{Acknowledgements}
This research was supported by National Natural Science Foundation of China Grants No. 11301130,  No. 11431010,
and Australian Research Council Discovery-Project Grant DP140103239.


\begin{thebibliography}{9999}

\bibitem{Be} Beck, Jonathan, Braid group action and quantum affine algebras. {\sl Comm. Math. Phys. \bf 165} (1994), no. 3, 555--568.

\bibitem{BGZ} Bracken, A. J.; Gould, M. D.; Zhang, R. B.,
Quantum supergroups and solutions of the Yang-Baxter Equation.
{\sl Modern Physics Letters \bf A5} (1990) no. 11, 831--840.

\bibitem{CP1} Chari, Vyjayanthi; Pressley, Andrew, Quantum affine algebras and their representations. {\sl Representations of groups} (Banff, AB, 1994), 59--78, CMS Conf. Proc., 16, Amer. Math. Soc., Providence, RI, 1995.

\bibitem{CP2} Chari, Vyjayanthi; Pressley, Andrew, Twisted quantum affine algebras. {\sl Comm. Math. Phys. \bf {196}} (1998), no. 2, 461--476.

\bibitem{De} Damiani, Ilaria, Drinfeld realisation of affine quantum algebras: the relations. {\sl Publ. Res. Inst. Math. Sci. \bf{ 48}} (2012), no. 3, 661--733.



\bibitem{Dr}  Drinfeld, V. G.,  A new realisation of Yangians and of quantum affine algebras. (Russian) {\sl Dokl. Akad. Nauk SSSR \bf {296} } (1987), no. 1, 13--17; translation in {\sl Soviet Math. Dokl. \bf{36}} (1988), no. 2, 212--216.

\bibitem{H} Heckenberger, Istvan; Spill, Fabian; Torrielli, Alessandro; Yamane, Hiroyuki,
 Drinfeld second realization of the quantum affine superalgebras of $D^{(1)}(2,1;x)$ via the Weyl groupoid. In {\em Combinatorial representation theory and related topics}, 171--216, RIMS Kokyuroku Bessatsu, B8, Res. Inst. Math. Sci. (RIMS), Kyoto, 2008.


\bibitem{Jn1}  Jing, Naihuan, Twisted vertex representations of quantum affine algebras. {\sl Invent. Math.\bf{102}} (1990), no. 3, 663--690.

\bibitem{Jn2} Jing, Naihuan, On Drinfeld realisation of quantum affine algebras. In {\em The Monster and Lie algebras} (Columbus, OH, 1996), 195--206, Ohio State Univ. Math. Res. Inst. Publ., 7, de Gruyter, Berlin, 1998.

\bibitem{JnM}  Jing, Naihuan; Misra, Kailash C., Vertex operators for twisted quantum affine algebras. {\sl Trans. Amer. Math. Soc. \bf{ 351}} (1999), no. 4, 1663--1690.

\bibitem{JZ} Jing, Naihuan; Zhang, Honglian, Drinfeld realisation of Quantum Twisted Affine Algebras via Braid Group. {\sl Adv. Math. Phys.} (2016), Art. ID 4843075, 15 pp.


\bibitem{WZ} Wu, Yuezhu; Zhang, R. B., Integrable representations of the quantum affine special linear superalgebra. {\sl
 Adv. Theor. Math. Phys. \bf{20}} (2016), no.3, 553-593.

\bibitem{XZ} Xu, Ying;  Zhang, R. B.,
Quantum correspondences of affine Lie superalgebras. {\sl Math. Research Lett.}, in press.  arXiv:1607.01142.

\bibitem{Y94} Yamane H.,  Quantized enveloping algebras associated with simple
Lie superalgebras and their universal $R$- matrices.
 {\sl Publ. RIMS. Kyoto Univ. \bf{30}} (1994), 15-87.

\bibitem{Y99} Yamane H.,  On definding relations of the affine Lie superalgebras
and their quantized universal enveloping superalgebras.
{\sl Publ. RIMS. Kyoto Univ. \bf{35}} (1999), 321-390.

\bibitem{Zh} Zhang, Huafeng,
Representations of quantum affine superalgebras. {\sl Math. Z. \bf 278}
 (2014), 663--703.

\bibitem{Z92a}  Zhang, R. B.,  Braid group representations arising from quantum supergroups
with arbitrary q and link polynomials.
{\sl J. Math. Phys. \bf 33} (1992), no. 11, 3918--3930.

\bibitem{Z92b}   Zhang, R. B.,  Finite-dimensional representations of $\U_q(osp(1/2n))$
and its connection with quantum $so(2n+1)$.
{\sl Lett. Math. Phys. \bf 25} (1992), no. 4, 317--325.

\bibitem{Z93} Zhang, R. B.,  Finite-dimensional irreducible representations of
the quantum supergroup  $\U_q(gl(m/n))$.
{\sl  J. Math. Phys. \bf 34} (1993), no. 3, 1236--1254.

\bibitem{Z95}  Zhang, R. B., Quantum supergroups and topological invariants of three-manifolds.
{\sl Rev. Math. Phys. \bf 7} (1995), no. 5, 809--831.

\bibitem{Z2} Zhang, R. B., Symmetrizable quantum affine superalgebras and their representations.
 {\sl J. Math Phys. \bf{38}} (1997),  535--543.


\bibitem{Z98}
Zhang, R. B.,  Structure and representations of the quantum general linear supergroup.
{\sl Commun. Math. Phys. \bf 195}  (1998)  525 -- 547.

\bibitem{ZBG91} Zhang, R. B.;  Bracken, A. J.;  Gould, M. D., 
Solution of the graded Yang-Baxter equation associated with
the vector representation of $\U_q(osp(M/2n))$.
{\sl Phys. Lett. \bf B 257} (1991), no. 1-2, 133-139.

\bibitem{ZGB91b}
Zhang, R. B.;   Gould, M. D.;  Bracken, A. J., 
Solutions of the graded classical Yang-Baxter equation and integrable models.
    {\sl J. Phys. \bf  A 24} (1991), no. 6, 1185--1197.

\end{thebibliography}
\end{document}